\documentclass{article}

\usepackage{arxiv}

\usepackage[utf8]{inputenc} % allow utf-8 input
\usepackage[T1]{fontenc}    % use 8-bit T1 fonts
\usepackage{hyperref}       % hyperlinks
\usepackage{url}            % simple URL typesetting
\usepackage{booktabs}       % professional-quality tables
\usepackage{amsfonts}       % blackboard math symbols
\usepackage{nicefrac}       % compact symbols for 1/2, etc.
\usepackage{microtype}      % microtypography
\usepackage{lipsum}
\usepackage{graphicx}
\graphicspath{ {./images/} }

\usepackage{amssymb}
\usepackage{amscd}
\usepackage{graphics,amsmath,amssymb}
\usepackage{amsthm}
\usepackage{latexsym}
\usepackage{epsf}
\usepackage{comment}

\theoremstyle{plain}
\newtheorem{theorem}{Theorem}

\newtheorem{lemma}[theorem]{Lemma}

\theoremstyle{definition}
\newtheorem{definition}[theorem]{Definition}

\theoremstyle{remark}
\newtheorem{remark}[theorem]{Remark}

\title{On the Existence of Balanced Generalized de Bruijn Sequences}

\author{
  Matthew Baker \\
  School of Mathematics\\
  Georgia Institute of Technology\\
  Atlanta, GA 30332\\
  \texttt{mbaker@math.gatech.edu}
  \And
  Bhumika Mittal \\
  Ashoka University\\
  Sonipat, Haryana 131029 \\
  \texttt{bhumika.mittal$\_$ug24@ashoka.edu.in} \\
  \And
  Haran Mouli \\
  PSBB KKN\\
  Chennai, Tamil Nadu 600078 \\
  \texttt{mouliharan@gmail.com} 
  \And
  Eric Tang \\
  Harvard University\\
  Cambridge, MA 02138 \\
  \texttt{etang@college.harvard.edu} \\
 \thanks{The second, third, and fourth authors completed this research under the direction of the first author as part of the summer 2021 PROMYS program. We thank the PROMYS program for their support and for creating such a stimulating work environment. We also thank Persi Diaconis for helpful feedback on an earlier draft. The first author was supported by a Simons Foundation Collaboration Grant.}

}

\begin{document}
\maketitle
\begin{abstract}
%Input expository part of abstract
A balanced generalized de Bruijn sequence with parameters $(n,l,k)$ is a cyclic sequence of $n$ bits such that (a) the number of 0's equals the number of 1's, and (b) each substring of length $l$ occurs at most $k$ times. 
We determine necessary and sufficient conditions on $n,l$, and $k$ for the existence of such a sequence.
% for the existence of a balanced generalized de Bruijn sequence with parameters $(n,l,k)$.
\end{abstract}

\section{Statement of the main theorem}

De Bruijn sequences, named after Nicolaas Govert de Bruijn (who first wrote about them in 1946) but first explored systematically
by Camille Flye Sainte-Marie in 1894, are well-studied in mathematical literature. A (binary) \textbf{de Bruijn sequence of order $m$} is a cyclic sequence where every possible $m$-bit substring occurs exactly once. It is well known (\cite{deBruijn}, see also \cite{Hall}) that there are $2^{2^{m-1} - m}$ distinct de Bruijn sequences of order $m$.
In this paper, we generalize de Bruijn sequences in the following manner:

\medskip

\begin{definition} 
A \textbf{generalized de Bruijn sequence} with parameters $(n,l,k)$ is a cyclic sequence of $n$ bits such that each substring of length $l$ occurs at most $k$ times. Such a sequence is called {\em balanced} if the number of $0$s and $1$s are equal.
\end{definition}

In our terminology, a (classical) de Bruijn sequence of order $m$ is a generalized de Bruijn sequence with parameters $(2^m,m,1)$. Such a sequence is automatically balanced.

The main result of this paper characterizes those parameters $(n,l,k)$ for which a balanced generalized de Bruijn sequence with parameters $(n,l,k)$ exists:

\medskip

\begin{theorem} \label{maintheorem}
Given positive integers $n,l$, and $k$, a balanced generalized de Bruijn sequence with parameters $(n,l,k)$ exists if and only if $n$ is even and $k \geqslant \frac{n}{2^l}$.
\end{theorem}

In Section~\ref{sec:motivation}, we discuss our original motivation for introducing balanced generalized de Bruijn sequences, which came from the world of mathematical card tricks.

\section{Proof of the main theorem}

Our proof of Theorem~\ref{maintheorem} used ideas from graph theory. Specifically, we will be concerned with certain subgraphs of the following family of directed graphs:

\begin{definition}
For $l \geqslant 2$, the \textbf{de Bruijn graph} of rank $l$, denoted by $G_l$, is a directed graph on $2^{l-1}$ vertices, labeled by the elements of $\{0,1\}^{l-1}$, and $2^l$ edges, labeled by the elements of $\{0,1\}^l$. For any $l$ bits $b_0,b_1,\cdots,b_{l-1}$, the edge $e=b_0b_1\cdots b_{l-1}$ connects $v_1 = b_0b_1\cdots b_{l-2}$ to $v_2 = b_1b_2\cdots b_{l-1}$.
\end{definition}

We begin with some simple lemmas.

\medskip

\begin{lemma}\label{inoutdeg}
Every vertex in the de Bruijn graph $G_l$ has both in-degree and out-degree $2$.
\end{lemma}

\begin{proof}
Consider the vertex $v = b_0b_1\cdots b_{l-2} \in V(G_l)$. By inspection, $v$ has in-edges $b b_0b_1\cdots b_{l-2}$ and out-edges $b_0b_1\cdots b_{l-2}b$ where $b \in \{0,1\}$. Thus, $v$ has both in-degree and out-degree $2$.
\end{proof}

\medskip

\begin{lemma}
The de Bruijn graph $G_l$ is Eulerian.
\end{lemma}

\begin{proof}
We know from Lemma \ref{inoutdeg} that every vertex in $G_l$ has equal in-degree and out-degree. Thus, by a well-known criterion for the existence of Eulerian circuits in digraphs (see for example~\cite[Theorem 13.1.2 ]{Brualdi}) it suffices to show that $G_l$ is connected. This follows from the fact that given vertices $v = b_0b_1\cdots b_{l-2}$ and $v' = b_0'b_1'\cdots b_{l-2}'$, the following is a walk from $v$ to $v'$:
\[b_0b_1\cdots b_{l-2} \to b_1\cdots b_{l-2}b_0' \to \cdots \to b_0'b_1'\cdots b_{l-2}' .\]
% Thus, $G_l$ is Eulerian.
\end{proof}

It will be convenient to color each edge of $G_l$ either red or blue according to the following rule:

\medskip

\begin{definition}
We define an edge in $G_l$ to be \textbf{red} if its last bit is $0$ and \textbf{blue} if its last bit is $1$. A subgraph of $G_l$ is \textbf{balanced} if it contains an equal number of red and blue edges.
\end{definition}

\medskip

\begin{lemma}\label{redblue}
Each vertex $v \in V(G_l)$ has one red and one blue out-edge, and the two in-edges of $v$ have the same color.
\end{lemma}
\begin{proof}
Let $v=b_0b_1\cdots b_{l-2}$ be a vertex in $G_l$. Then by definition, the out-edges $b_0b_1\cdots b_{l-2}0$ and $b_0b_1\cdots b_{l-2}1$ are red and blue, respectively. The in-edges $0b_0b_1\cdots b_{l-2}$ and $1b_0b_1\cdots b_{l-2}$ must end with the same bit, and are thus of the same color.
\end{proof}

Recall that a \textbf{circuit} in a graph is a path which begins and ends at the same vertex, and a \textbf{cycle} is a circuit in which no vertex (other than the initial and final one) is repeated.

By symmetry, the deBruijn graph $G_l$ itself is balanced for all $l \geqslant 2$. We are interested in certain balanced subgraphs of $G_l$ because of the following simple observation: 
% The following lemmas will pave the way to the proof of our main theorem.

\medskip

\begin{lemma}\label{sev}
There is a one-to-one correspondence between balanced circuits of length $n$ in $G_l$ and balanced generalized de Bruijn sequences with parameters $(n,l,1)$.
\end{lemma}

\begin{proof}
Let the edges of the circuit of length $n$ be denoted as: 
\[b_0b_1\cdots b_{l-1} \to b_1b_2\cdots b_l \to \cdots \to b_{n-1}b_0\cdots b_{l-2} . \]
The $l$-bit substrings of the $n$-bit sequence $b_0b_1\cdots b_{n-1}$ precisely correspond to the edges of the circuit of length $n$. Thus, as the edges on the circuit are distinct, so are the $l$-bit substrings. Moreover, since the number of red and blue edges in the circuit are equal, the number of $0$s and $1$s in the string are equal as well. Hence, the $n$-bit string is a balanced generalized de Bruijn sequence with parameters $(n,l,1)$.
\end{proof}

\begin{lemma}\label{cycle}
If the graph $G_l$ has a balanced circuit of length $n$, then the graph $G_{l+1}$ has a balanced cycle of length $n$.
\end{lemma}

\begin{proof}
The edges of $G_l$ and vertices of $G_{l+1}$ have the same labelling. This gives us a natural bijection between the edges of $G_l$ and the vertices of $G_{l+1}$. Let $e_1,e_2 \in G_l$ correspond to $v_1,v_2 \in G_{l+1}$ under this bijection. It follows from the definition of a de Bruijn graph that the head of $e_1$ is the tip of $e_2$ in $G_l$ if and only if there is an edge connecting $v_1$ to $v_2$ in $G_{l+1}$. Now, consider the edges of the circuit of length $n$ in $G_l$:
\[b_0b_1\cdots b_{l-1} \to b_1b_2\cdots b_l \to \cdots \to b_{n-1}b_0\cdots b_{l-2} . \]
We consider these edges as the vertices of a walk in $G_{l+1}$. As all the edges of the circuit are distinct, the walk in $G_{l+1}$ contains distinct vertices. Thus, the bijection maps a circuit of length $n$ in $G_l$ to an $n$-cycle in $G_{l+1}$. Moreover, since the circuit is balanced, there are equal number of $0$s and $1$s among $b_0, b_1, \cdots, b_{n-1}$. These will be the last bits in the edges of the cycle in $G_{l+1}$ as well, and hence, the cycle is balanced.
\end{proof}

\begin{lemma}\label{nin}
Let $l \geqslant 2$, and let $n$ be an even positive integer at most $2^l$. Then the graph $G_l$ contains a balanced circuit of length $n$.
\end{lemma}

\begin{proof}
We proceed by induction on $l$. First, for $l=2$, we can either have $n=2$ or $n=4$. We may take the circuits $0 \to 1$ and $0 \to 1 \to 1 \to 0$, respectively. 

Now, assume that the claim holds true for $l$. We prove it holds true for $l+1$. For even $n \leqslant 2^l$, we may use the induction hypothesis to find a balanced circuit of length $n$ in $G_l$. From Lemma \ref{cycle}, there exists a balanced cycle (also a circuit) of length $n$ in $G_{l+1}$. 

We now show the claim for $2^l < n < 2^{l+1}$ where $n$ is even. Since $2 \leqslant 2^{l+1}-n < 2^l$, there exists a balanced $n$-cycle in $G_{l+1}$, say $H$. Then the graph $G_{l+1}-H$ contains $n$ edges. Assume that $G_{l+1}-H = H_1 \cup H_2 \cup \cdots \cup H_t$, where each $H_i$ for $1 \leqslant i \leqslant t$ is a component of $G_{l+1}-H$. Since both $G_{l+1}$ and $H$ are balanced and Eulerian, $G_{l+1}-H$ must be balanced and each $H_i$ must be Eulerian.

If $G_{l+1}-H$ is connected, we are done, since it is a circuit in $G_{l+1}$ of length $n$. Otherwise, since $G_{l+1}$ is connected, there exists an edge $e$ in $H$ which connects vertices in two different components of $G_{l+1}-H$, say $H_1$ and $H_2$. Let the edge $e$ go from $v_1 \in V(H_1)$ to $v_2 \in V(H_2)$. Without loss of generality, assume $e$ is red. Now, let $e_1$ be an edge from $v_1$ to $u_1$ in $H_1$, and let $e_2$ be an edge from $u_2$ to $v_2$ in $H_2$. Then the last $l-1$ bits of $v_1$ and $u_2$ are the first $l-1$ bits of $v_2$ and $u_1$. Hence, there must be an edge $e'$ from $u_2$ to $u_1$ in $H$. We deduce from Lemma \ref{redblue} that since $e$ is red, $e_1$ must be blue, $e_2$ must be red, and $e'$ must be blue.

Consider the new subgraph obtained by replacing $e_1$ and $e_2$ with $e$ and $e'$. One checks easily that the degrees of all vertices, as well as the number of red and blue edges, are preserved. Moreover, the number of connected components has been reduced. Thus, we end up with a balanced subgraph with $n$ edges and fewer components, all of which are Eulerian. By repeating this process, we will eventually end up with a balanced subgraph of $n$ edges which is a circuit. Finally, for $n=2^{l+1}$, all of the edges of the graph form a circuit since $G_{l+1}$ is Eulerian. The result follows.
\end{proof}

With these preliminaries out of the way, we now turn to the proof of our main result:
% Theorem~\ref{maintheorem}:

\medskip

% \begin{theorem}
% For any positive integers $n,l,k$, a balanced generalized de Bruijn sequence with parameters $(n,l,k)$ exists if and only if $n$ is even % and $k \geqslant \frac{n}{2^l}$.
% \end{theorem}

\begin{proof}[Proof of Theorem~\ref{maintheorem}]
$(\Rightarrow)$ Since a balanced generalized de Bruijn sequence has an equal number of $0$s and $1$s, $n$ must be even. Moreover, by the Pigeonhole Principle, we must have $k \geqslant \frac{n}{2^l}$, since there are only $2^l$ possible $l$-bit substrings and $n$ total substrings.\medskip

$(\Leftarrow)$ Now, we show that these constraints are sufficient for a balanced generalized de Bruijn sequence to exist. We may assume that $k = \lceil \frac{n}{2^l} \rceil$ since a balanced generalized de Bruijn sequence with parameters $(n,l,k)$ is also one with parameters $(n,l,k')$ where $k'>k$. \medskip

For $l=1$, the claim is trivial since the substrings are simply bits, and we may use any balanced string of length $n$. Now, let $l \geqslant 2$. We will now prove the claim by induction on $k$. The base case $k=1$ follows from Lemmas \ref{sev} and \ref{nin}. It suffices to show that the existence of a balanced generalized de Bruijn sequence with parameters $(n,l,k)$ implies the existence of one with parameters $(n+2^l,l,k+1)$. Let $b_0b_1\cdots b_{n-1}$ be a balanced generalized de Bruijn sequence with parameters $(n,l,k)$, and let $b_0'b_1'\cdots b_{2^l-1}'$ be one with parameters $(2^l,l,1)$. The latter contains every $l$-bit string as a substring exactly once. We may therefore assume that $b_0'b_1'\cdots b_{l-1}' = b_0b_1\cdots b_{l-1}$. Thus
\[b_0b_1\cdots b_{n-1}b_0'b_1'\cdots b_{l-1}'\]
is a balanced sequence in which each substring of length $l$ appears at most $k$ times in the first part of the sequence and exactly once in the second part. It is therefore a balanced generalized de Bruijn sequence with parameters $(n+2^l,l,k+1)$. The theorem follows by induction.
\end{proof}

\begin{remark}
Theorem~\ref{maintheorem} can be generalized as follows.
Say that a (finite) binary sequence is \textbf{almost-balanced} if the number of 0's and the number of 1's differ by \textit{at most one}. Since the complement of an almost-balanced sequence is also almost-balanced, the same argument used to prove Theorem~\ref{maintheorem} yields the following:
\end{remark}

\begin{theorem}\label{genmainresult}
Given positive integers $n,l$, and $k$, an almost-balanced generalized de Bruijn sequence with parameters $(n,l,k)$ exists if and only if $k \geqslant \frac{n}{2^l}$.
\end{theorem}

\section{Motivation} \label{sec:motivation}

% {\bf Trick \#1.}

One of our motivations for introducing balanced generalized de Bruijn sequences comes from the following card trick, a 32-card version of which is described on p.18 of \cite{Diaconis-Graham}. The 52-card variant described here is alluded to on p.29 of {\em loc. cit.}

\medskip

{\bf Effect and presentation.}

A regular deck of 52 cards, secured with rubber bands in its case, is tossed out into the audience. The deck is tossed around a few times until finally one audience member is instructed to remove the rubber bands, open the case, and give the deck a cut at some random position. This audience member takes the top card and then passes it to another spectator who removes the current top card. This is repeated until five members of the audience have taken cards. The performer instructs the five participants who have taken cards to look at their card, make a mental picture of it, and try to transmit an image of their card to the performer. The performer says, ``You're all doing great, but some of your signals are crossing and I can't quite make everything out. Would everyone who has a red card please stand up and try extra hard to project the mental image of their card?'' Turning to the first participant that took a card, the performer says, ``I'm having trouble making out the suit of your card\ldots it's not a heart, is it?'' The participant shakes his head no. ``No, I didn't think so\ldots no, it's a diamond. It's the 9 of Diamonds!'' The participant nods in agreement. Without asking any more questions, the performer goes on to identify the exact card that each of the other four participants is thinking of.

\medskip

{\bf Method and explanation.}

The trick is based on a balanced generalized de Bruijn sequence with parameters $(52,5,2)$, for example the following one:
\[
0000011101010010001011001101111100000101101111101001
\]

Each 5-bit string corresponds to either one or two cards in a standard deck of 52 cards; the color of the corresponding card is red if the first bit is a 0 and black if the first bit is a 1. 

When the five spectators signal the color of their respective cards, this produces a 5-bit string (0 = red, 1 = black, going in order from spectator 1 to spectator 5). The magician looks up this 5-bit string in the following table:
\[
\begin{array}{llllllll}
00000: & A\heartsuit, A\diamondsuit  & 01000: & 10\diamondsuit,  K\heartsuit & 10000: & 4\clubsuit,4\spadesuit & 11000: & 7\spadesuit \\
00001: & 7\heartsuit, 7\diamondsuit  & 01001: & K\diamondsuit, 5\heartsuit & 10001: & 6\clubsuit & 11001: & J\spadesuit \\
00010: & 6\heartsuit, 6\diamondsuit  & 01010: & 2\heartsuit & 10010: & 3\clubsuit,A\spadesuit & 11010: & 8\spadesuit, J\clubsuit \\
00011: & 3\diamondsuit & 01011: & 9\heartsuit, 9\diamondsuit & 10011: & 5\clubsuit & 11011: & A\clubsuit,2\spadesuit \\
00100: & 5\diamondsuit, 10\heartsuit & 01100: & J\heartsuit & 10100: & 7\clubsuit, 3\spadesuit & 11100: & 10\clubsuit \\
00101: & 8\diamondsuit, 8\heartsuit & 01101: & Q\heartsuit,4\diamondsuit & 10101: & 10\spadesuit  & 11101: & K\spadesuit, 8\clubsuit \\
00110: & 3\heartsuit & 01110: & 2\diamondsuit & 10110: & Q\clubsuit, Q\spadesuit & 11110: & 9\spadesuit, K\clubsuit \\
00111: & Q\diamondsuit & 01111: & 4\heartsuit,J\diamondsuit & 10111: & 6\spadesuit, 2\clubsuit & 11111: & 5\spadesuit, 9\clubsuit \\
\end{array}
\]

If there are two possibilities listed, for example 9H or 9D, the magician feints hesitation, saying (in this example) to the first spectator, ``I'm having trouble seeing the suit\ldots it's not a heart, is it?'' If the spectator says yes, the magician responds, ``Yes, I thought so! It's the 9 of Hearts!'' If the spectator says no, the magician responds, ``No, I didn't think so\ldots it's a diamond. It's the 9 of Diamonds!''

The magician now knows the first spectator's card, and is able to reveal all of the other selections by finding the first spectator's card in the following list (the next four cards -- wrapping around cyclically, if necessary -- are, in that order, the cards which the next four spectators are holding):
\[
\begin{aligned}
A\heartsuit \; 7\heartsuit \; 3\diamondsuit \; Q\diamondsuit \; 2\diamondsuit \; K\spadesuit \; 8\spadesuit \; 10\spadesuit \; 2\heartsuit \; 7\clubsuit \; K\diamondsuit \; 3\clubsuit \; 5\diamondsuit \; 10\diamondsuit \; 6\clubsuit \; 6\heartsuit \; 8\diamondsuit \; 9\heartsuit \; Q\clubsuit \; J\heartsuit \; J\spadesuit \; 5\clubsuit \; 3\heartsuit \; Q\heartsuit \; A\clubsuit \; 2\clubsuit \\
4\diamondsuit \; 5\spadesuit \; 9\spadesuit \; 10\clubsuit \; 7\spadesuit \; 4\clubsuit \; A\diamondsuit \; 7\diamondsuit \; 6\diamondsuit \; 8\heartsuit \; 9\diamondsuit \; Q\spadesuit \; 4\heartsuit\; 2\spadesuit \; 6\spadesuit \; J\diamondsuit \; 9\clubsuit \; K\clubsuit \; 8\clubsuit \; J\clubsuit \;  \; 3\spadesuit \; 5\heartsuit \; A\spadesuit \; 10\heartsuit \; K\heartsuit \; 4\spadesuit \\ 
\end{aligned}
\]

To disguise the glimpsing of the necessary information, once the spectators have signaled the color of their cards, grab a clipboard and a pen and say that you're going to try to write down the exact cards held by each of the five participants. Unbeknownst to the participants, on the clipboard you have a crib sheet containing the above information.

\newpage

\section{Some open problems}

% {\bf List some more open problems here}

\begin{enumerate}
    \item In lieu of binary bits, consider an alphabet $A$ of size $m$ where $m>2.$ Does a suitable analogue of Theorem~\ref{maintheorem} still hold? That is, given positive integers $n,l,k,$ and $m$, does an $A$-balanced generalized deBruijn sequence of length $n$ exist if and only if $m$ divides $n$ and $k\geq \frac{n}{m^l}$?
    As considering the complement of such a sequence no longer suffices, the induction step used to prove Theorem~\ref{maintheorem} does not apply in this more general context.
    \item Suppose $n$ is even and $k \geqslant \frac{n}{2^l}$. How many balanced generalized de Bruijn sequences with parameters $(n,l,k)$ are there? We know there is at least 1, and when  $(n,l,k) = (2^m,m,1)$ there are precisely $2^{2^{m-1} - m}$ such sequences. Is there an easily described general formula -- at least in some special cases?
    \item Are there algebraic `shift register' type constructions that would permit the performer to do the $(52,5,2)$ trick described in Section 3 without mnemonics or a list? The trick described on p.18 of \cite{Diaconis-Graham} (with parameters $(32,5,1)$) has this property. For analogous constructions related to a 6-card version of our $(52,5,2)$ trick, see \cite{Lawson-Perfect} or \cite{Song}.
\end{enumerate}
%%%%%%%%%%%%%%%%%%%%%%%%%%%%%%%%%%%%%%%%%%%%%%%%%%%%%%%%%%%%%%%%%%%%%%

\bibliographystyle{alpha}
\bibliography{references}

%%%%%%%%%%%%%%%%%%%%%%%%%%%%%%%%%%%%%%%%%%%%%%%%%%%%%%%%%%%%%%%%%%%%%%
\end{document}